\documentclass[aps,pre,eqsecnum,superscriptaddress,11pt,onecolumn]{revtex4-1}

\usepackage[letterpaper,margin=1in]{geometry} % Set 1-inch margins on US letter-size paper
\usepackage{color}
\usepackage{hyperref}
\usepackage{amsmath}
\usepackage{amsfonts}
\usepackage{graphicx}
\usepackage{natbib}
\usepackage{listings}
\usepackage{wrapfig}
\usepackage{mathtools}  
\usepackage{amssymb}
\usepackage{tabulary}
\usepackage{caption}
\usepackage{subcaption}
\usepackage{booktabs}
\usepackage{amsthm}
\newtheorem{theorem}{Theorem}
\newtheorem{corollary}{Corollary}
\newtheorem{lemma}{Lemma}
\usepackage[export]{adjustbox}

\begin{document}

% Use the \preprint command to place your local institutional report
% number in the upper righthand corner of the title page in preprint mode.
% Multiple \preprint commands are allowed.
% Use the 'preprintnumbers' class option to override journal defaults
% to display numbers if necessary
%\preprint{}

%Title of paper
\title{The Number of Spanning Trees in some Special Self-Similar Graphs}

% repeat the \author .. \affiliation  etc. as needed
% \email, \thanks, \homepage, \altaffiliation all apply to the current
% author. Explanatory text should go in the []'s, actual e-mail
% address or url should go in the {}'s for \email and \homepage.
% Please use the appropriate macro foreach each type of information

\author{M. A. Morsy}
\email{s-malaamorsy@zewailcity.edu.eg}
\affiliation{University of Science and Technology in Zewail City of Science and Technology, 6th of October City, Giza, Egypt.}
\author{M. Anwar}
\affiliation{Department of Mathematics, Pennsylvania State University, University Park, PA 16802, USA.}
\affiliation{Department of Mathematics, Faculty of Science, Ain Shams University, Cairo, Egypt.}
\author{A. W. Aboutahoun}
\affiliation{Applied Mathematics and Information Science Department, Zewail City of Science and Technology, 6th of October City, Giza, Egypt.}
\affiliation{Department of Mathematics, Faculty of Science, Alexandria University, Alexandria, Egypt.}

%Collaboration name if desired (requires use of superscriptaddress
%option in \documentclass). \noaffiliation is required (may also be
%used with the \author command).
%\collaboration can be followed by \email, \homepage, \thanks as well.
%\collaboration{}
%\noaffiliation

\begin{abstract}
In this paper, we introduce two families of planar and self-similar graphs which have small-world properties. The constructed models are based on an iterative process where each step of a certain formulation of modules results in a final graph with a self-similar structure. The number of spanning trees of a graph is one of the most graph-theoretical parameters, where its applications range from the theory of networks to theoretical chemistry. Two explicit formulas are introduced for the number of spanning trees for the two models. With explicit formulas for some of their topological parameters as well. 

Keywords: \textit{Self-Similar Graphs, Cycle Graphs, Wheel Graphs, Spanning Trees, Clustering Coefficients, Entropy of Graph}. 
\end{abstract}

% insert suggested keywords - APS authors don't need to do this
%\keywords{}

\maketitle

\section{Introduction} \label{sec:intro}

Having information about the number of spanning trees is not just valuable for mathematics; however, it is also valuable for many different fields, in physics, it is beneficial with regular lattice, Sierpinski gaskets, and more, it has a direct relation to the probability theory, related to the chemistry through the chemical isomers, in biology, social sciences, and of course, networks, as the number of spanning tree is the crucial measure of network reliability \cite{Anema_2016}. According to \cite{DHAR200629, PhysRevLett.64.1613}, the number of spanning trees of a connected network is precisely equal to the number of recurrent configurations of the Abelian sand-pile model on the network, which is equivalent to the chip-firing game 17 under certain constraints and is a paradigm for self-organized criticality.

Enumerating the number of spanning trees in finite graphs is not a recent problem. It started over 100 years ago as one of the graph's most critical and valuable invariants. In 1847, Kirchhoff came up with the theorem of matrix-tree that can compute the number of the spanning tree concerning the determinant of the Laplacian matrix of the graph, and the Laplacian matrix calculated through the subtraction of the adjacency matrix of the graph from its degree matrix $L(G) = D(G) - A(G)$. The Degree matrix is a diagonal matrix with the degree of each vertex on the diagonals.  

However, counting the number of spanning trees using this method for large graphs will be complicated. Therefore, finding a way to compute the spanning tree without Kirchhoff's determinant will be helpful. 

Many attempts try to compute this, especially the trial of \cite{ELATIK2021106117}, which successfully calculated the number of spanning trees in a class of self-similar graphs. However, it is for only one case of cycle graphs; and there can be so many cycles and wheels with $n$ number of vertices and fractals that contain $m$ as the length of the path in fractals. Moreover, many types of graphs, such as complete graphs or general graphs overall, still need to be studied. This paper solves the problem of counting the spanning trees of the self-similar graphs constructed by starting with $C_n$ and $W_n$, with specific operations and general $m$ for the length path. Moreover, it gives insight into the entropy of the self-similar graphs, how they got affected by these operations of fractals, and studies some topological properties of the constructed graphs. 

In the remainder of this paper, section \ref{sec:preliminary} introduces the needed preliminary for this paper, setting the definitions provided with the needed figures, and the table of notations. Section \ref{sec:CN} talks about the first models which are the self-similar graphs based on cycle graphs; sub-section(A) is for how the fractals or the self-similar graph grows, sub-section(B) is for the theorem of the number of spanning trees for this model, sub-section(C) is for the Entropy of this model's graphs, and sub-section(D) is for the clustering coefficients for this model's graphs. Section \ref{sec:WN} is the order of section \ref{sec:CN} but the second models which are the self-similar graphs based on wheel graphs. 

\section{Preliminary} \label{sec:preliminary}

In this section, we set some important requirements and definitions before the main theorem and results, and firstly, we stick to the terms and definitions of Bondy and Murty \cite{bondy_murty_1976}. All graphs $G(V(G), E(G))$ of the paper are simple graphs, and for all e in $E(G)$, there exists $u,v$ such that: 
\begin{align*}
    \Psi(G) &: E(G) {\longrightarrow} (V(G){\times}V(G)) \\
    \psi(G) &: e {\longrightarrow} (u,v)    
\end{align*}

\textbf{Definition 1.} A \textit{spanning tree} of any given graph is a sub-graph of the graph which is connected and has no cycles. It is a tree if it has no cycles. Graph $G$ is connected if it contains spanning tree $\tau$ because any two vertices have a connection by a path in $\tau$, and hence in $G$. However, if $G$ is a connected graph that is not a tree and $e$ is an edge of one of its cycles. Then $(G{\setminus}e)$ is a spanning sub-graph of $G$ that is likewise connected since an edge of a graph $G$ is a cut edge if and only if $e$ belongs to no cycle of $G$. We get a spanning tree of $G$ by repeatedly eliminating edges in cycles until the only edge left is a cut-edge. As a result, a graph is only considered connected if and only if it has a spanning tree.

\textbf{Definition 2.} An operation is called \textit{Edge-Path-Transformation} and shortly as $\xi_{1}$ or $(EPT)$, if we replace every edge $e$ in $E(G)$ with path of length $m$ by adding $(m-1)$ vertices, as in the figure \ref{Fig:1}, where $\xi_{1}$ is applied on $C_4$, with $m = 2$. Moreover, this leads to having a central graph donated by $H^{(i,m)}$ where $i$ is the number of times this operation has been applied to the graph, so $H^{(i,m)}$ is the graph obtained by this operation.
\begin{figure}[h]
    \centering
    \includegraphics[width=0.2\textwidth, height=0.07\textwidth]{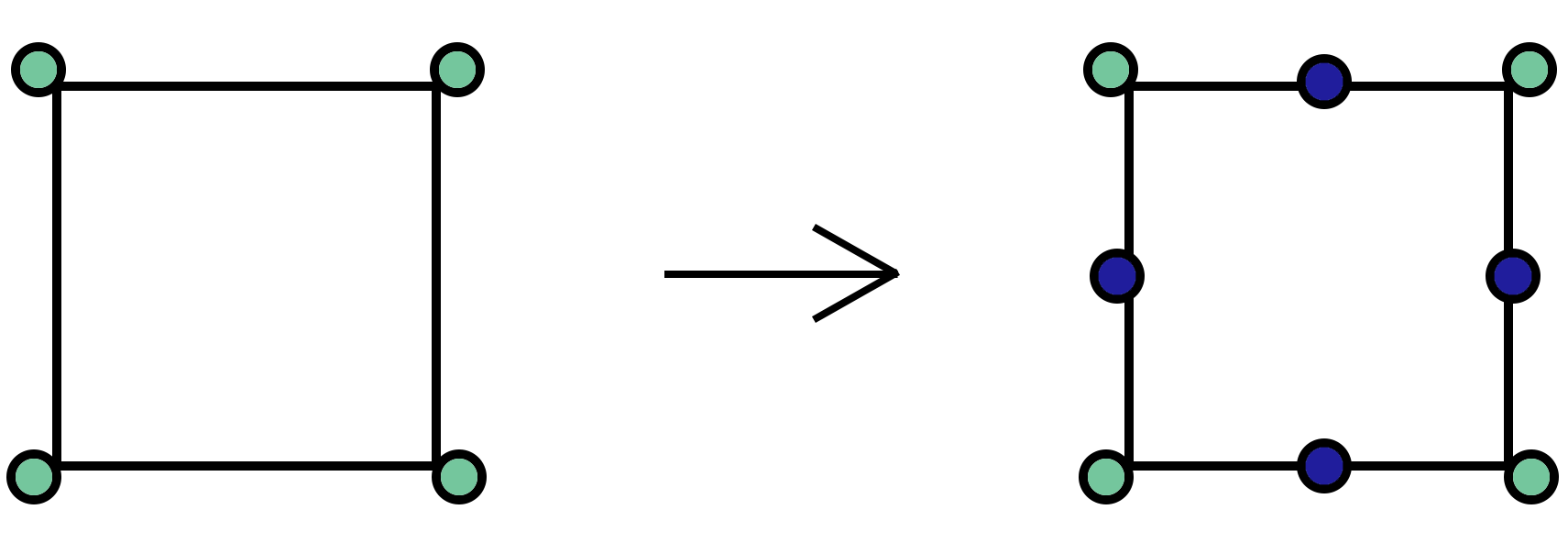}
    \caption{$(EPT)$ is applied on $C_4$ with $m = 2$.}
    \label{Fig:1}
\end{figure} 

\textbf{Definition 3.} An operation is called \textit{Graph-linking-Vertex} and shortly as $\xi_{2}$ or $(GLV)$, if we attach every vertex $v$ in the graph $G^{(i,m,n)}$ to a new same graph $G^{(0,m,n)}$, as in the figure \ref{fig:2}.
\begin{figure}[h]
    \centering
    \includegraphics[width=0.2\textwidth, height=0.1\textwidth]{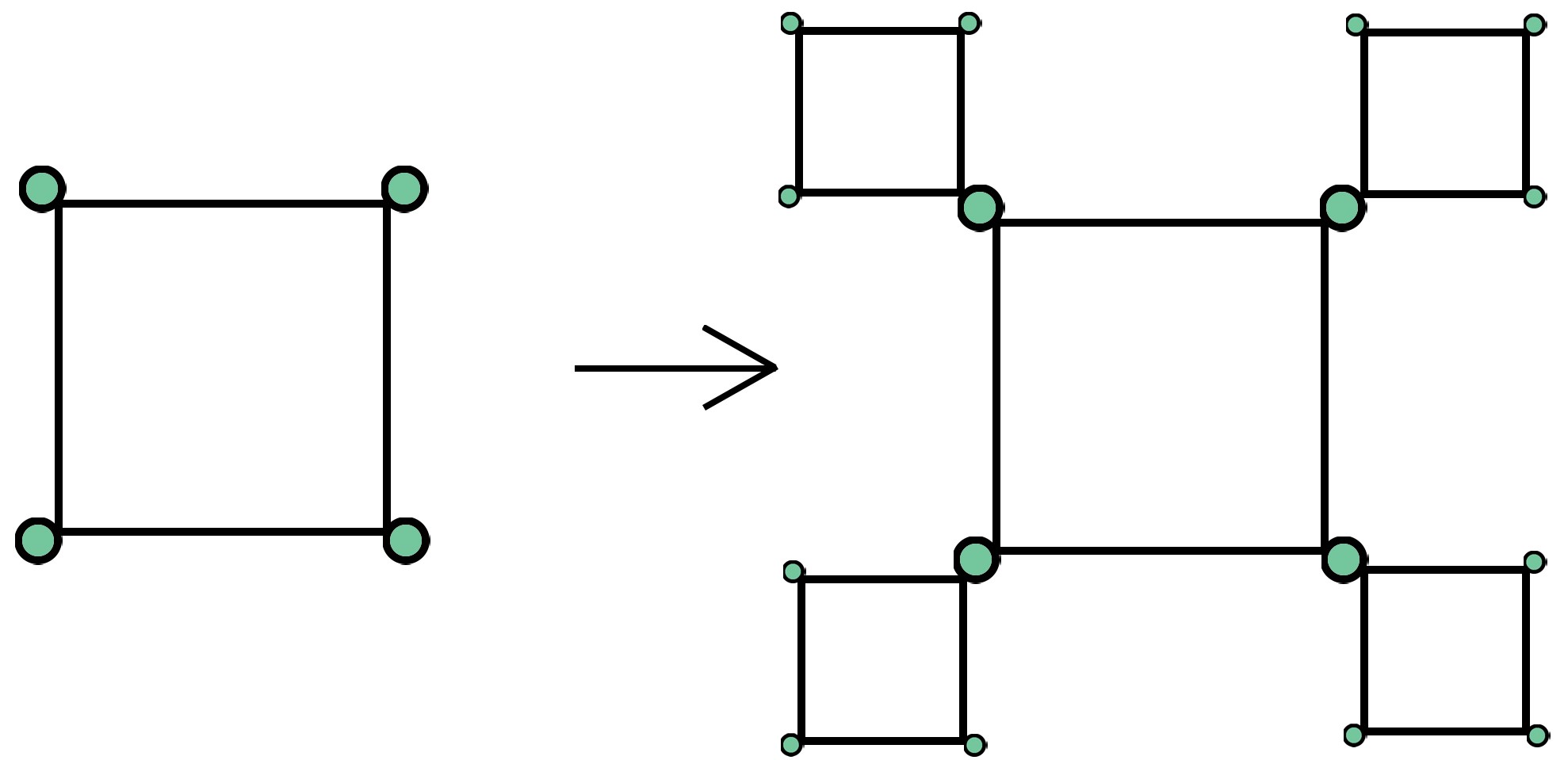}
    \caption{$(GLV)$ is applied on $C_4$.}
    \label{fig:2}
\end{figure}

\textbf{Definition 4.} $G_{(V^{(i+1)},E^{(i+1))}}^{(i,m,n)}$, and $i \geq -1$, is the graph that contains $V^{(i+1)}$ of vertices and $E^{(i+1)}$ of edges. For example; in the case of cycles, that it begins from {$G^{(0,2,4)}$ is the case with $V^{(1)} = 4$, $E^{(1)} = 4$ which is $C_4$, and then by applying the two operations, as in figure \ref{Fig:3}, where they are applied on $C_4$. 
\begin{figure}[h]
    \centering
    \includegraphics[width=0.2\textwidth, height=0.2\textwidth]{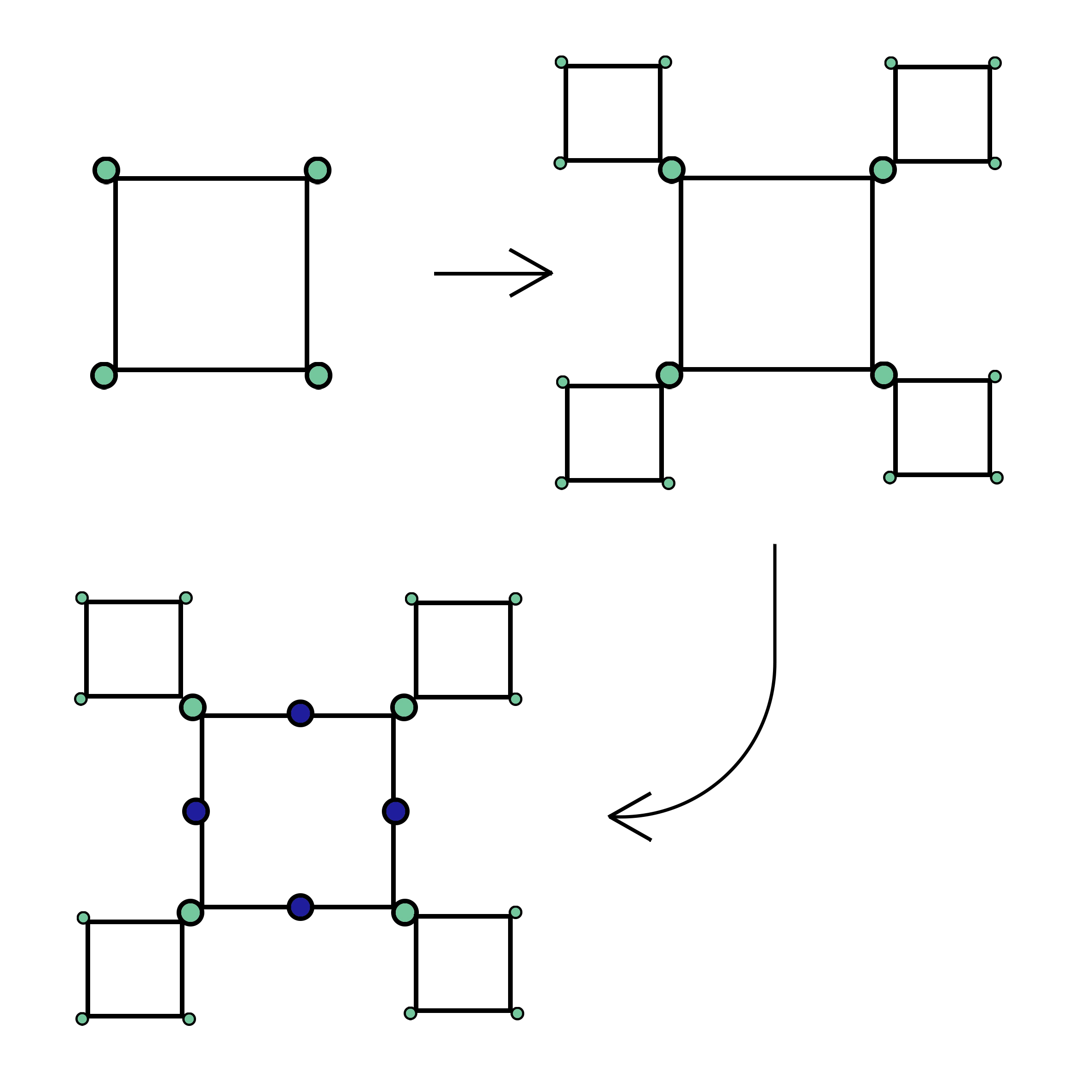}
    \caption{$(EPT)$ and $(GLV)$ is applied on $C_4$ with $m = 2$.}
    \label{Fig:3}
\end{figure}
\begin{table}[h]
    \caption{Table of Notations}
    \tiny
    \vspace{-1.1em}
    \centering
        \begin{tabular}{c c p{0.5\textwidth}} 
        \hline \hline 
            $m$ & $\triangleq$ & The length of the path in the operation of $\xi_{1}$. \\
            $G^{(i,m)}_{C_{n}}$ & $\triangleq$ & The self similar graph based on $C_{n}$ or shortly $G^{(i,m,n)}$ in the section III. \\
            $G^{(i,m)}_{W_{n}}$ & $\triangleq$ & The self similar graph based on $W_{n}$ or shortly $G^{(i,m,n)}$ in the section IV. \\
            $G^{(0,m,n)}$ & $\triangleq$ & The base graph whether $W_{n}$, or $C_{n}$ according to the section. \\
            $G^{(-1,m,n)}$ & $\triangleq$ & Graph with just one vertex and without any edges. \\
            $V^{(i+1)}(G^{(i,m,n)})$, $E^{(i+1)}(G^{(i,m,n)})$ & $\triangleq$ & The number of vertices, and edges in self-similar graph $G^{(i,m,n)}$ in stage $i$. \\
            $H^{(i,m)}$ & $\triangleq$ & The central graph in the $i$ stage in graph $G^{(i,m)}$, which is the result of operating on the $G^{(0,m)}$ graph, by the operation $\xi_{1}$, $i$ times. \\
            $\tau(G^{(i,m,n)})$ & $\triangleq$ & The total number of spanning trees of the graph $G^{(i,m,n)}$. \\
            $\sigma(G^{(i,m,n)})$ & $\triangleq$ & The entropy of the self-similar $G^{(i,m,n)}$ whether for the model based on $W_{n}$ or $C_{n}$ according to the section. \\
            $\bar{\mathcal{A}}(G^{(i,m,n)})$ & $\triangleq$ & The Average clustering coefficient of the self-similar $G^{(i,m,n)}$ whether for the model based on $W_{n}$ or $C_{n}$ according to the section.
            \end{tabular}
        \label{tab:TableOfNotations}
\end{table} 

\newpage

\section{The self-similar (Fractals) graphs based on $C_n$ graphs} \label{sec:CN}

\subsection{The Fractal Growth}

The fractals based on the $C_n$ graphs are constructed by the operations in \textbf{Definition 2, 3}, which is $\xi_{1}$ and  $\xi_{2}$ with given $m$. 
\noindent \textbf{Example 1:} For instance, beginning with $G^{(0)} = C_3$, and $m = 2$.
\begin{table}[h] \label{tab:C4}
  \begin{tabular}
      {|c|c|c|} \hline Stage & Graphs inside & Graph \\
      \hline $G^{(0,2,3)}$ & $C_{3}$ & {
      \includegraphics[width=0.5in]{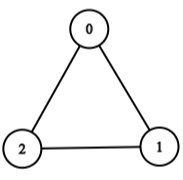}} \\
      $G^{(1,2,3)}$ & $3 G^{(0,2,3)} \cup H^{(1,2)}$ & {
      \includegraphics[width=0.8in]{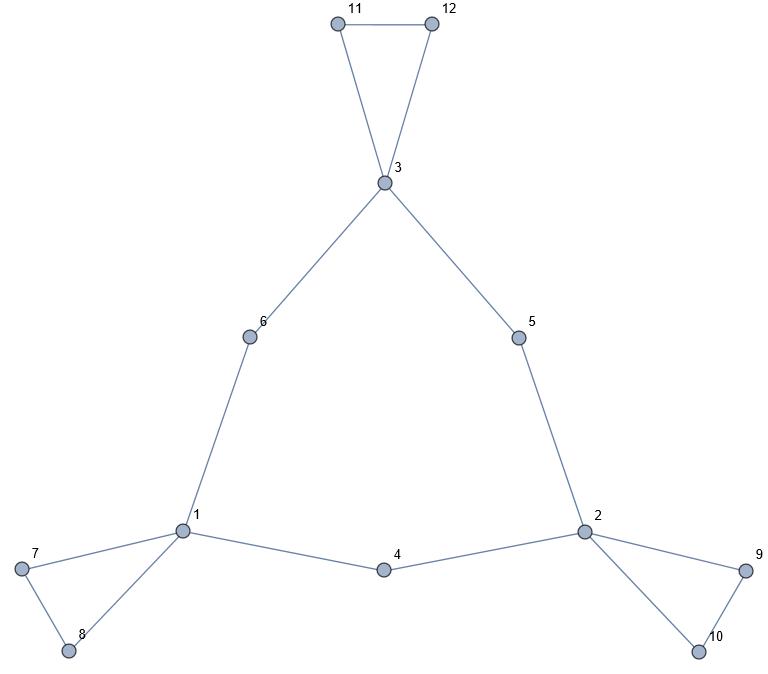}} \\
      $G^{(2,2,3)}$ & $3 G^{(1,2,3)} \cup 3 G^{(0,2,3)} \cup H^{(2,2)}$ & {
      \includegraphics[width=0.8in]{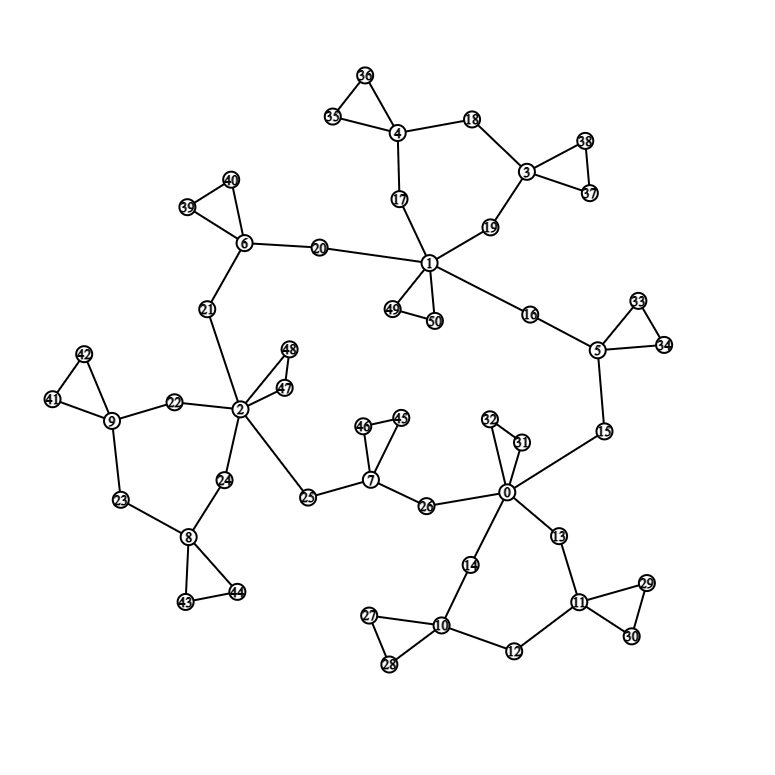}} \\
      \hline
  \end{tabular}
\end{table}
From the previous example, we can observe that $G^{(1,2,3)}$ is $3$ copies of $G^{(0,2,3)}$ and $H^{(1,2)}$, and $G^{(2,2,3)}$ is $3 \times 2^0$ copies of $G^{(1,2,3)}$ and $3$ copies of $G^{(0,2,3)}$ and $H^{(2,2)}$. Therefore, each copy of $G^{(t,m,n)}$ in $G^{(i,m,n)}$ will be transformed to $G^{(t+1,m,n)}$ in $G^{(i+1,m,n)}$, and the number of copies will be preserved, which is illustrated in the following diagram: 
{
\footnotesize
\begin{align*}
    i & \longrightarrow i+1 \\
    {\#}G^{(i-1,m,n)} & \longrightarrow {\#}G^{(i,m,n)} \\
    {\#}G^{(i-2,m,n)} & \longrightarrow {\#}G^{(i-1,m,n)} \\
    {\#}G^{(i-3,m,n)} & \longrightarrow {\#}G^{(i-2,m,n)} \\
    {\#}G^{(i-4,m,n)} & \longrightarrow {\#}G^{(i-3,m,n)} \\
    {\#}G^{(i-5,m,n)} & \longrightarrow {\#}G^{(i-4,m,n)} \\
    &\;\;\vdots \notag \\
    {\#}G^{(1,m,n)} & \longrightarrow {\#}G^{(2,m,n)} \\
    {\#}G^{(0,m,n)} & \longrightarrow {\#}G^{(1,m,n)} \\
    V^{(i)} \subseteq V(H^{(i,m)}) & \longrightarrow {\#}G^{(0,m,n)} \\
    H^{(i,m)} & \longrightarrow H^{(i+1,m)} \\
\end{align*}} \begin{lemma} \label{lemma 1}
    In $G^{(i,m,n)}$, we have ($n$ copies of $G^{(i-1,m,n)}$, $n\times m^0$ copies of $G^{(i-2,m,n)}$, $n\times m^1$ copies of $G^{(i-3,m,n)}$, $n\times m^2$ copies of $G^{(i-4,m,n)}$, ... , $n\times m^{(k-2)}$ copies of $G^{(i-k,m,n)}$, and $H^{(i,m)} = C_{n\times m^{i}}$).
\end{lemma} 
\begin{proof}
    Lemma \ref{lemma 1} is true when (i = 1, and i = 2), which is in example 1. Assume that in $G^{(i,m,n)}$, we have ($n$ copies of $G^{(i-1,m,n)}$, $n\times m^0$ copies of $G^{(i-2,m,n)}$, $n\times m^1$ copies of $G^{(i-3,m,n)}$, $n\times m^2$ copies of $G^{(i-4,m,n)}$, ... , $n\times m^{(k-2)}$ copies of $G^{(i-k,m,n)}$, and $H^{(i,m)} = C_{n\times m^{i}}$), and we want to prove that in $G^{(i+1,m,n)}$, we have ($n$ copies of $G^{(i,m,n)}$, $n\times m^0$ copies of $G^{(i-1,m,n)}$, $n\times m^1$ copies of $G^{(i-2,m,n)}$, $n\times m^2$ copies of $G^{(i-3,m,n)}$, ... , $n\times m^{(k-1)}$ copies of $G^{(i+1-k,m,n)}$, and $H^{(i+1,m)} = C_{n\times m^{i+1}}$) when we perform the operations of $\xi_{1}$ and $\xi_{2}$, we can see that each copy of $G^{(i,m,n)}$ will be converted to $G^{(i+1,m,n)}$. 

    Since each copy of $G^{(t,m,n)}$ in $G^{(i,m,n)}$ will be transformed to $G^{(t+1,m,n)}$ in $G^{(i+1,m,n)}$, and the number of copies will be preserved, and a subset of vertices in $H^{(i,m)}$ will be transformed to $G^{(0,m,n)}$, and those subsets are the vertices that have been added to the edges in the operation of $\xi_{1}$, and the $H^{(i,m)}$ will be transformed to $H^{(i+1,m)}$.

    Therefore, we can assure that in $G^{(i+1,m,n)}$ that contains ($n$ copies of $G^{(i,m,n)}$, $n\times m^0$ copies of $G^{(i-1,m,n)}$, $n\times m^1$ copies of $G^{(i-2,m,n)}$, $n\times m^2$ copies of $G^{(i-3,m,n)}$, ... , $n\times m^{(k-2)}$ copies of $G^{(i+1-k,m,n)}$, and $H^{(i+1,m)} = C_{n\times m^{(i+1)}}$) when we perform the operations of $\xi_{1}$ and $\xi_{2}$, we can see that each copy of $G^{(i,m,n)}$ will be converted to $G^{(i+1,m,n)}$. Moreover, we have central graph $H^{(i,m)}$ which is $C_{n\times m^{i}}$ will grow $H^{(i+1,m)}$ which is $C_{n\times m^{i+1}}$, so it will grow in this way using the operations in definition 2, 3, to be the graph in the definition 4. 
\end{proof}
\textbf{Example 2:} In the figure \ref{Fig. 4} the first iteration for $C_3$ by applying those operations, therefore, $G_{(V^{(2)},E^{(2)})}^{(1,2,3)}$, where $G_{(V^{(1)},E^{(1)})}^{(0,2,3)} = C_3$ with considering $V^{(1)} = 3$ and $E^{(1)} = 3$, the figure \ref{Fig. 5}, the second iteration for $C_3$ which is $G_{(V^{(3)},E^{(3)})}^{(2,2,3)}$, thus from the observation: $V^{(2)} = 12$ and $E^{(2)} = 15$, and $V^{(3)} = 51$ and $E^{(3)} = 66$. Those two operations will yield that, for every existing vertex in the graph, it gives $n$ edges and $(n-1)$ vertices in the next iteration. Moreover, for every existing edge in the graph, it gives $m$ edges and $(m-1)$ vertices. 

\begin{figure}[h]
    \centering
    \includegraphics[width=0.2\textwidth, height=0.2\textwidth]{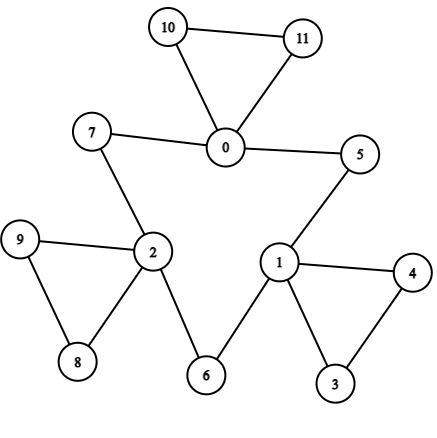}
    \caption{The first iteration for $C_3$ where $m = 2$.}
    \label{Fig. 4}
\end{figure}
\begin{figure}[h]
    \centering
    \includegraphics[width=0.3\textwidth, height=0.3\textwidth]{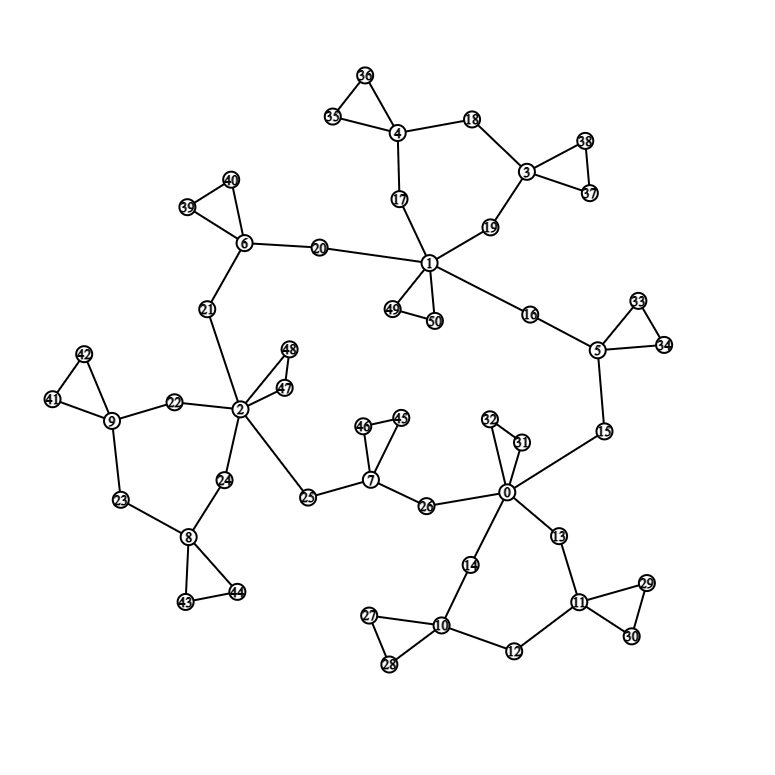}
    \caption{The second iteration for $C_3$ where $m = 2$.}
    \label{Fig. 5}
\end{figure}

\begin{lemma} \label{lemma2}
    The number of vertices and edges in any fractals based on $C_n$ graphs that were built previously can be given by the following two recursive formulas:
\begin{gather*}
    \vert{V^{(i)}}\vert = (n)\times {\vert}V^{(i-1)}\vert + (m-1)\times {\vert}E^{(i-1)}\vert \\
    {\vert}{E^{(i)}}{\vert} = (n)\times {\vert}V^{(i-1)}{\vert} + (m)\times {\vert}E^{(i-1)}{\vert}       
\end{gather*}
And by solving them together, it is easy to separate them as follows:
\begin{gather} \label{verteq}
    {\vert}{V^{(i)}}{\vert} = (n+m)\times {\vert}V^{(i-1)}{\vert} - (n)\times {\vert}V^{(i-2)}{\vert} \\
    {\vert}{E^{(i)}}{\vert} = (n+m)\times {\vert}E^{(i-1)}{\vert} - (n)\times {\vert}E^{(i-2)}{\vert}   
\end{gather}
Where the $n$ indicates the original graph vertices and the $m$ is the path length in the $\xi_{1}$. 
\end{lemma}
Therefore, $V^{(1)} = 3$, $E^{(1)} = 3$, and we will take further initial step that $V^{(0)} = 1$, $E^{(0)} = 0$, those considered as a constants. The terms of the first recursive formula which is for the number of vertices in stage $i$ can be obtained from this equation (Which is the equation form for the recursive relation in the equation \ref{verteq}): 

\small
\begin{equation} \label{eq3.3}
        {\vert} V^{(i)}{\vert} = \frac{2^{-i-1} \left((\text{$\alpha_1$}+\varphi ) (\alpha_2 -\varphi )^i+(\varphi -\text{$\alpha_1$}) (\varphi +\alpha_2 )^i\right)}{\varphi }
\end{equation}

\normalsize
\noindent Where the constants are: $\varphi = \sqrt{-4 n + (m + n)^2}$, 
$\alpha_2 = m+n$, $\alpha_1 = m-n$, ($i > 0$, $m \geq 2$, $n \geq 3$).

\subsection{The Number of Spanning Tree for self-similar (Fractal) graphs based on $C_n$ graphs}

\begin{theorem}
    The number of spanning trees for any $C_{n}$ in any stage $i \geq 0$ that is constructed in the same described way in the paper, can be given by the following equation:
    \begin{equation} \label{eq3.4}
        \tau(G^{(i,m,n)}) = (n^{\sum^{i}_{j=0}{{\vert}V^{(j)}{\vert}}})(m^{\sum^{i}_{j=0}{(i-j)\times{\vert}V^{(j)}{\vert}}})
    \end{equation}
\end{theorem}
\begin{proof}
    By using Lemma \ref{lemma 1}, and \ref{lemma2}; we can assure that we have ($n$ copies of $G^{(i-1,m,n)}$, $n\times m^0$ copies of $G^{(i-2,m,n)}$, $n\times m^1$ copies of $G^{(i-3,m,n)}$, $n\times m^2$ copies of $G^{(i-4,m,n)}$, ... , $n\times m^{(k-2)}$ copies of $G^{(i-k,m,n)}$, and $H^{(i,m)} = C_{n\times m^{i}}$) in the $G^{(i,m,n)}$ graph, and also, the $G^{(i-1,m,n)}$ contains similarly, ($n$ copies of $G^{(i-2,m,n)}$, $n\times m^0$ copies of $G^{(i-3,m,n)}$, $n\times m^1$ copies of $G^{(i-4,m,n)}$, $n\times m^2$ copies of $G^{(i-5,m,n)}$, ... , $n\times m^{(k-3)}$ copies of $G^{(i-k-1,m,n)}$, and $H^{(i-1,m)} = C_{n\times m^{i-1}}$). 
    
    Therefore, recursively, we will have the $G^{(0,m,n)}$ number of times, that is equivalent to ${\vert}V^{(i)}{\vert}$, so $\tau(G^{(0,m,n)}) = \tau(C_{(n)}) = n$, as it repeated number of times and it is known, so $(\tau(G^{(0,m,n)}))^{{\vert}V^{(i)}{\vert}} = \tau(C_{(n)})^{{\vert}V^{(i)}{\vert}} = (n)^{{\vert}V^{(i)}{\vert}}$, and the central graph which is growing in the way that $H^{(i,m)} = C_{(n\times m^i)}$ the number of spanning tree is $\tau(H^{(i,m)}) = \tau(C_{n\times m^{i}})^{{\vert}V^{(i-1)}{\vert}} = (n\times m^{i})^{{\vert}V^{(i-1)}{\vert}} = n^{{\vert}V^{(i-1)}{\vert}}m^{i\times{\vert}V^{(i-1)}{\vert}}$, Finally, by having $k$ of $H^{i,m}$ that is according to lemma \ref{lemma 1} is $H^{(i,m)}$, $H^{(i-1,m)}$, $H^{(i-2,m)}$, ... , $H^{(1,m)}$, hence, we well have summation in the power as: 
    \begin{equation} \label{eg3.5}
        \tau(G^{(i,m,n)}) = n^{\sum^{i}_{j=0}{{\vert}V^{(j)}{\vert}}}m^{\sum^{i}_{j=0}{(i-j)\times{\vert}V^{(j)}{\vert}}}
    \end{equation}
    And by executing the summations, using \ref{eq3.3}, we got: 
    \small
    \begin{equation*}
        \sum^{i}_{j=0}{{\vert}V^{(j)}{\vert}} = {\frac{2^{-i} \left(-n (\alpha_2 -\varphi )^i+n (\varphi +\alpha_2 )^i+2^i \varphi \right)}{\varphi }}
    \end{equation*}
    \begin{equation*}
        \begin{split}
            \sum^{i}_{j=0}{(i-j)\times{\vert}V^{(j)}{\vert}} & = \frac{1}{(m-1) \varphi} \\
            & {\times}m^{ \varphi  (i m-i-n)} \\
            & {\times}m^{2^{-i-1} (\alpha_2 -\varphi )^i \left(m n+n^2+n \varphi -2 n\right)} \\
            & {\times}m^{2^{-i-1} (\varphi +\alpha_2 )^i \left(-m n-n^2+n \varphi +2 n\right)}
        \end{split}
    \end{equation*}
\end{proof}

\normalsize
\noindent \textbf{Example 3:} Taking $n = 3$, and $m = 2$ as in \cite{ELATIK2021106117}
\begin{equation} \label{eq.3.6}
    \begin{split}
        \tau\left(G_{(C_{3})}^{\left(i,2\right)}\right) & = 3^{-\frac{3 \left(\frac{1}{2} \left(5-\sqrt{13}\right)\right)^i}{\sqrt{13}}+\frac{3 \left(\frac{1}{2} \left(\sqrt{13}+5\right)\right)^i}{\sqrt{13}}+1} \\
        & {\times} 2^{\frac{3 \left(\sqrt{13}+3\right) 2^{-i-1} \left(5-\sqrt{13}\right)^i}{\sqrt{13}}} \\
        & {\times} 2^{\frac{3 \left(\sqrt{13}-3\right) 2^{-i-1} \left(\sqrt{13}+5\right)^i}{\sqrt{13}}} \\
        & {\times} 2^{(-3 + i)}
    \end{split}
\end{equation}

\noindent In this case, the recursive formula for ${\vert}V^{(i)}{\vert}$ will be: 
\begin{align*}
        {\vert}{V^{(i)}}{\vert} &= (5){{\vert}V^{(i-1)}{\vert}} - (3){{\vert}V^{(i-2)}{\vert}} \\
        {\vert}V^{(i)}{\vert} &= \frac{2^{-i-1} \left((\text{$\alpha $1}+\varphi ) (\alpha_2 -\varphi )^i+(\varphi -\text{$\alpha $1}) (\varphi +\alpha_2 )^i\right)}{\varphi}  
\end{align*}

\noindent With substituting $\varphi$, and $\alpha_2$ by the n and m. 

The first few terms are: ${(1, 3, 12, 51, 219, 942, ...)}$, the number of spanning trees in this case are given by the following: 

\begin{equation*}
\begin{split}
\tau\left(G^{\left(1,2,3\right)}\right) & = 3^{{\vert}V^{1}{\vert}+{\vert}V^0{\vert}}\times2^{{0\times {\vert}V^{1}{\vert} + 1 \times {\vert}V^0{\vert}}} \\
 & = 3^{3+1}\times2^1 \\ 
 & = 3^4\times2^1
\end{split}
\end{equation*}
\begin{equation*}
\begin{split}
\tau\left(G^{\left(2,2,3\right)}\right) & = 3^{{\vert}V^2{\vert}+{\vert}V^1{\vert}+{\vert}V^0{\vert}}\times2^{{\vert}V^1{\vert}+2{\vert}V^0{\vert}} \\
 & = 3^{12+3+1}\times2^{3+2} \\
 & = 3^{16}\times2^5
\end{split}
\end{equation*}
\begin{equation*}
\begin{split}
\tau\left(G^{\left(3,2,3\right)}\right) & = 3^{{\vert}V^3{\vert}+{\vert}V^2{\vert}+{\vert}V^1{\vert}+{\vert}V^0{\vert}}\times2^{{\vert}V^2{\vert}+2{\vert}V^1{\vert}+3{\vert}V^0{\vert}} \\ 
 & = 3^{51+12+3+1}\times2^{12+6+3} \\
 & = 3^{67}\times2^{21}
\end{split}
\end{equation*}

Which can be written in another way that will illustrate the growth more clearly, as the following: 
\begin{equation*}
\begin{split}
\tau\left(G^{\left(1,2,3\right)}\right) & = {\tau\left(C_3\right)}^3\times\tau\left(C_6\right) \\
 & = 3^3\times6 \\
 & = 3^4\times2^1
\end{split}
\end{equation*}
\begin{equation*}
\begin{split}
\tau\left(G^{\left(2,2,3\right)}\right) & = {\tau\left(C_3\right)}^{12}\times{\tau\left(C_6\right)}^3\times\tau\left(C_{12}\right) \\
 & = 3^{12}\times6^3\times12 \\
 & = 3^{16}\times2^{5}
\end{split}
\end{equation*}
\begin{equation*}
\begin{split}
\tau\left(G^{\left(3,2,3\right)}\right) & = {\tau\left(C_3\right)}^{12}\times{\tau\left(C_6\right)}^3\times\tau\left(C_{12}\right) \\
 & = 3^{12}\times6^3\times12 \\
 & = 3^{67}\times2^{21}
\end{split}
\end{equation*}

\begin{table}[ht]
\caption{Comparing the results of $n = 3$, and $m = 2$} % title of Table
\centering % used for centering table
\begin{tabular}{c c c c} % centered columns (4 columns)
\hline\hline %inserts double horizontal lines
Iteration & \cite{ELATIK2021106117} & Comp. Res. & Our Res. \\ [0.5ex] % inserts table
%heading
\hline % inserts single horizontal line
1 & $3^4{\times}2^1$ & $3^4{\times}2^1$ & $3^4{\times}2^1$ \\ % inserting body of the table
2 & $3^{16}{\times}2^{5}$ & $3^{16}{\times}2^{5}$ & $3^{16}{\times}2^{5}$ \\
3 & $3^{67}{\times}2^{21}$ & $3^{67}{\times}2^{21}$ & $3^{67}{\times}2^{21}$ \\
4 & $3^{286}{\times}2^{88}$ & $3^{286}{\times}2^{88}$ & $3^{286}{\times}2^{88}$ \\ [1ex] % [1ex] adds vertical space
\hline %inserts single line
\end{tabular}
\label{table:results for n = 3 and m = 2} % is used to refer this table in the text
\end{table} The computational results is performed using a (\textit{Python}). 

\subsection{The Entropy in self-similar graphs (Fractals) based on $C_n$ graphs}

The entropy of a spanning tree which is known as the asymptotic complexity of the fractals is calculated explicitly from the presentation and the study of \cite{10.1214/aop/1176989121, Mokhlissi2018, lyons_2005}, it is an interesting term with a finite number that describes the network model. It can be given using the following formula: 
\begin{equation}\label{eq3.7}
    \lim_{i\to \infty}\left(\frac{ln(\tau(G^{(i,m,n)}))}{{\vert}V^{(i)}{\vert}}\right)
\end{equation}
\begin{corollary}
    Using the formula for spanning trees in equation \ref{eg3.5}, and having equation \ref{eq3.3} for the number of vertices, therefore, we get this: 

\small
\begin{equation}
    \sigma(G^{(i,m,n)}) = \frac{2 (m-1) n \ln (n)-n \ln (m) (-\varphi +\alpha_2 -2)}{(m-1) (\varphi -\text{$\alpha_1$})}
\end{equation}
\end{corollary}
\normalsize 

With conditions that $n>m$, $n\geq 3$, and $m\geq 2$.

Therefore, as a correction of the entropy in \cite{ELATIK2021106117}, we find that, in the case of $n = 3$, and $m = 2$, the entropy is not 0.396176 as obtained, but it is $\sigma = \frac{\ln \left(\frac{729}{512}\right)+\sqrt{13} \ln (8)}{\sqrt{13}+1} = 1.70465$. 

And plot in \ref{Fig. 6}, with $n$ are on the x-axis and $m$ is on y-axis, z-axis is the entropy.
\begin{figure}[h]
    \centering
    \includegraphics[width=0.2\textwidth, height=0.2\textwidth]{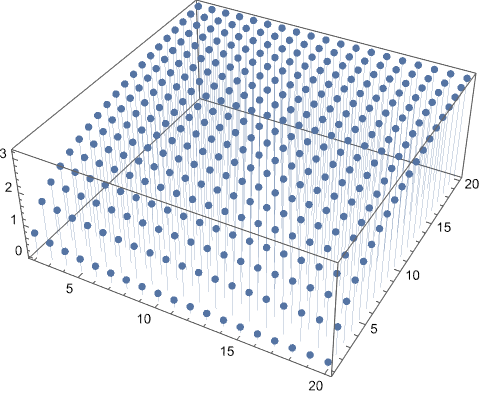}
    \caption{Plot of entropy of cycle fractals versus m and n.}
    \label{Fig. 6}
\end{figure}
\subsection{The Clustering Coefficient for self-similar graphs (Fractals) based on $C_n$ graphs}

The clustering coefficient is a measure of the degree to which vertices in a graph tend to cluster together. It is a common metric used to characterize the structure of graphs, particularly in the context of social networks. In an undirected graph, the clustering coefficient of a vertex is defined as the fraction of pairs of its neighbors that are also connected. 

and it can be calculated for $v_{i}$ by the following equation: 

\begin{equation}\label{eq3.9}
    C_{i}={\frac  {2|\{e_{{jk}}:v_{j},v_{k}\in N_{i},e_{{jk}}\in E\}|}{k_{i}(k_{i}-1)}} = \frac{2 \mathcal{E}_{i}}{k_{i}(k_{i}-1)}
\end{equation}

Where $\mathcal{E}_{i}$ is the total number of edges between the neighbors of $v_{i}$, and $k_{i}$ is the degree of this vertex. And the average clustering coefficient is just the average of the total of clutering coefficients. 

One important property of the clustering coefficient is that it tends to be higher in graphs with a more regular structure, such as lattices and small-world networks, and lower in graphs with a more random structure, such as random graphs. This has important implications for the spread of information or diseases in a network, as vertices with a higher clustering coefficient are more likely to be connected to other vertices in their cluster and thus more likely to transmit the information or disease to other members of the cluster.

There has been a significant amount of research on the clustering coefficient in graphs, particularly in the context of social networks. For example, \cite{Watts1998} showed that small-world networks, which have both high clustering and high degree of separation, can arise from relatively simple rewiring rules applied to regular lattices.

On the other hand, \cite{doi:10.1073/pnas.0400087101} found that the clustering coefficient in real-world social networks tends to decay with the degree of the vertex, a phenomenon known as the "disassortative mixing" pattern.

In addition, \cite{doi:10.1137/S003614450342480} developed a method to measure the global clustering coefficient of a graph, which gives a summary statistic for the entire graph rather than just individual vertices.

\textbf{Example 4:} In the graph $G^{(1,2,3)}$ which is described in the \ref{Fig. 4}, it has 12 vertex and the clustering coefficients can be easily obtained for each vertex, therefore, the set of clustering coefficients is ${\{\frac{1}{6}, \frac{1}{6}, \frac{1}{6}, 1, 1, 0, 0, 0, 1, 1, 1, 1}\}$ according to the number of vertices in the figure. therefore the average clustering coefficients can be obtained by: 

\begin{equation*}
    \frac{\frac{3}{6} + 6}{12} = \frac{\frac{3}{\binom{4}{2}} + 2\times{\vert}V^{(1)}{\vert}}{{\vert}V^{(2)}{\vert}}
\end{equation*}

\textbf{Example 5:} In the \ref{Fig. 5} the fractal graph $G^{(2,2,3)}$, it has 51 vertex and the clustering coefficients for all of them can be easily obtained for each vertex and average clustering coefficients can be obtained by: 

\begin{equation*}
    \frac{\frac{3}{15} + \frac{9}{6} + 12}{51} = \frac{\frac{3}{\binom{6}{2}} + \frac{{\vert}V^{(2)}{\vert} - {\vert}V^{(1)}{\vert}}{\binom{4}{2}} + 2\times{\vert}V^{(2)}{\vert}}{{\vert}V^{(3)}{\vert}}
\end{equation*}

\begin{theorem}
    The average clustering coefficient for any self-similar graph based on graphs of $C_n$ equals zero in $n\geq 4$, and with $n = 3$, the average clustering coefficient in stage $i \geq 1$ is as follows: 

\small
\begin{equation} \label{eq3.10}
    \bar{\mathcal{A}}(G^{(i,m,n)}) = \frac{\frac{3}{\binom{2(i+1)}{2}} + \sum^{i-1}_{j=1}\frac{{\vert}V^{(j + 1)}{\vert} - {\vert}V^{(j)}{\vert}}{\binom{2(i - j +1)}{2}} + 2\times {\vert}V^{(i)}{\vert}}{{\vert}V^{(i + 1)}{\vert}}
\end{equation}
\end{theorem}
\begin{proof}
    The case when $i = 1$, and $i = 2$ is illustrated in examples 4 and 5. In the fractals based on graphs of $C_n$ which is $G^{(i,m,n)}$ where $i \geq 1$, as $G^{(0,m,n)}$ is just the graph of $C_{3}$. The degree is continuously increasing as multiples of $2$. And by using the formula in equation \ref{eq3.9} $\frac{2}{k_{i}(k_{i}-1)} = \frac{1}{\binom{k_{i}}{2}}$, where $k_{i}$ is the degree. As proved in how the fractals are growing, it has proven that the number of graphs that transform from $G^{(i,m,n)}$ to $G^{(i+1,m,n)}$ is preserved. Therefore, the number of vertices is growing as described in the formula \ref{verteq}, and the degree is simply changed as $(0, 2, 4, 6, ..., 2i)$ where zero is the degree that did not exist in the stage $G^{(i-1,m,n)}$, that has been added by the operation $\xi_{1}$. 
    
    Furthermore, the degree transforms to a certain number of edges that can be described by ${\vert}V^{(j + 1)}{\vert} - {\vert}V^{(j)}{\vert}$ and there will be always the number of vertices that has a degree 2, and if they are not part of central graphs $H^{(i,m)}$, they will have a clustering coefficient equals to 1, and the number of them is 2 per each vertex. As for calculating the average, the denominator should be the total number of vertices in this stage ${i}$. Therefore, the formula in \ref{eq3.10} is correct.
\end{proof}

\section{The self-similar (Fractals) graphs based on $W_n$ graphs} \label{sec:WN}

\subsection{The fractal growth}

The fractals based on the $W_n$ graphs are constructed by the operations in \textbf{Definition 2, 3}, which is $\xi_{1}$ and  $\xi_{2}$ with given $m$, similar to the way for $C_n$. 

Taking into consideration that the wheel graph is defined as the following: $W_{n}(V, E)$, where $V = n+1$, and $E = 2n$. 

\noindent \textbf{Example 6:} For instance, beginning with $G^{(0,2,4)}$ is $W_4$, and $m = 2$.

\begin{table}
  [ht] \label{tab:C4}
  \begin{tabular}
      {|c|c|c|} \hline Stage & Graphs inside & Graph \\
      \hline $G^{(0,2,4)}$ & $W_{4}$ & {
      \includegraphics[width=1in]{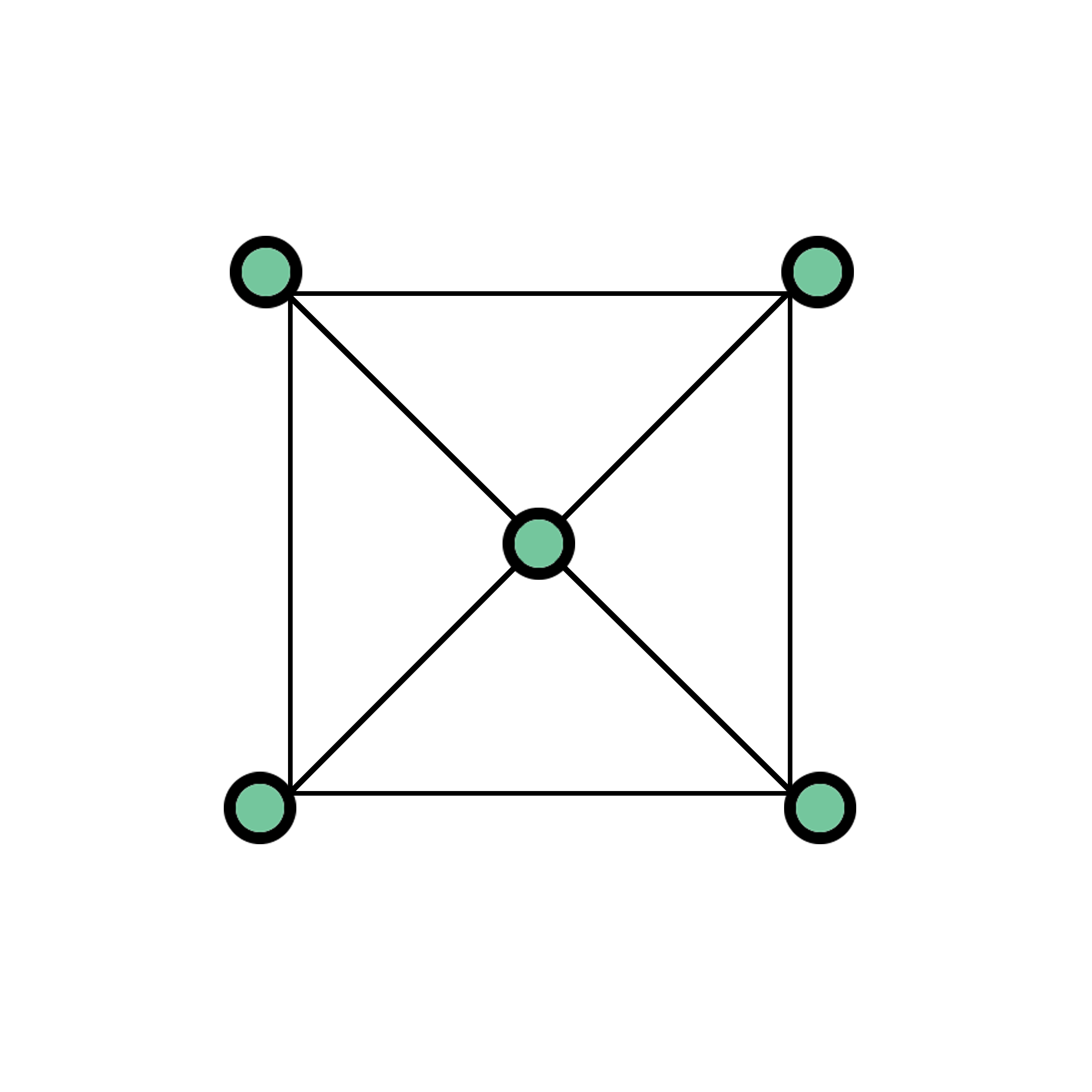}} \\
      $G^{(1,2,4)}$ & $5 G^{(0,2,4)} \cup H^{(1,2)}$ & {
      \includegraphics[width=1in]{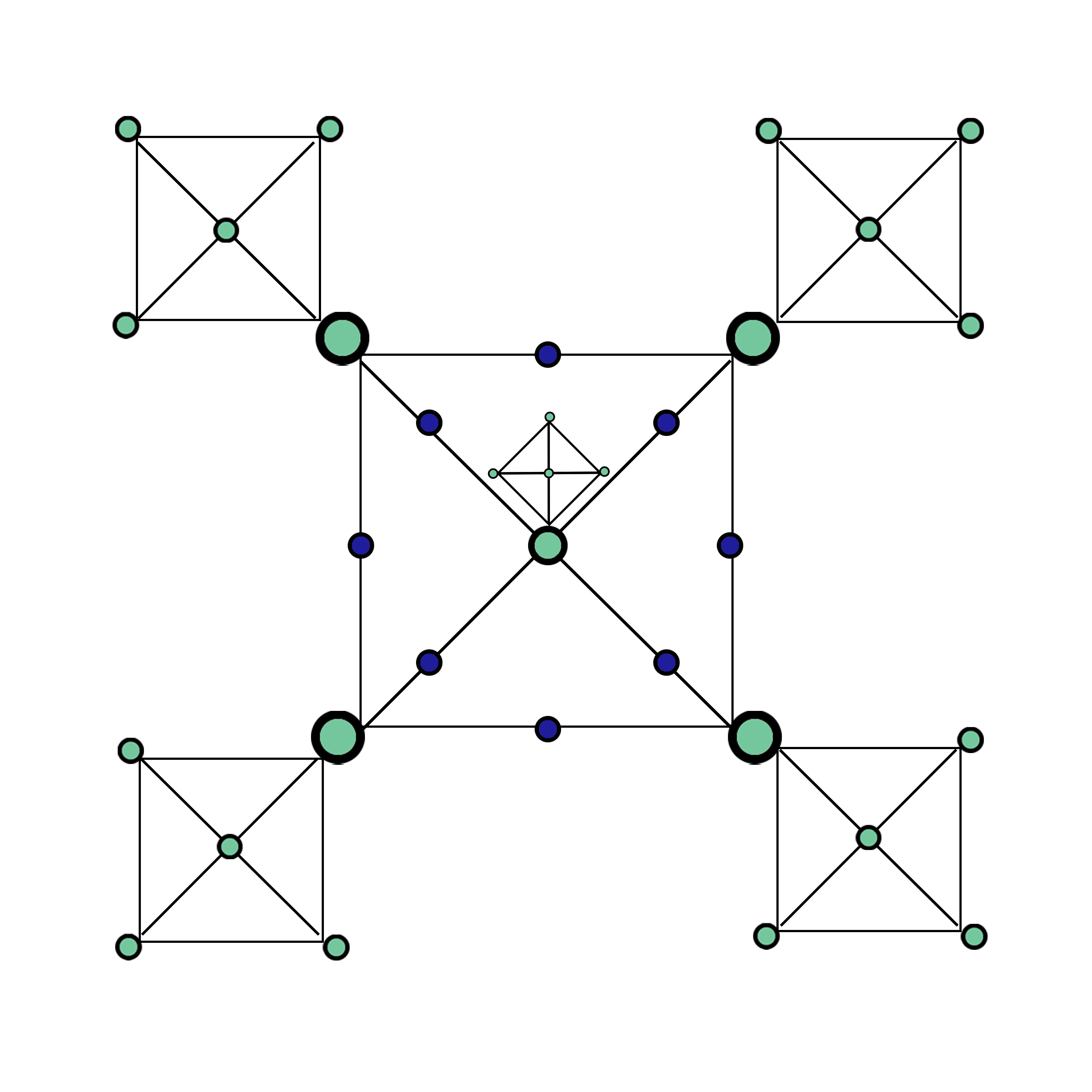}} \\
      $G^{(2,2,4)}$ & $5 G^{(1,2,4)} \cup 8 G^{(0,2,4)} \cup H^{(2,2)}$ & {
      \includegraphics[width=1in]{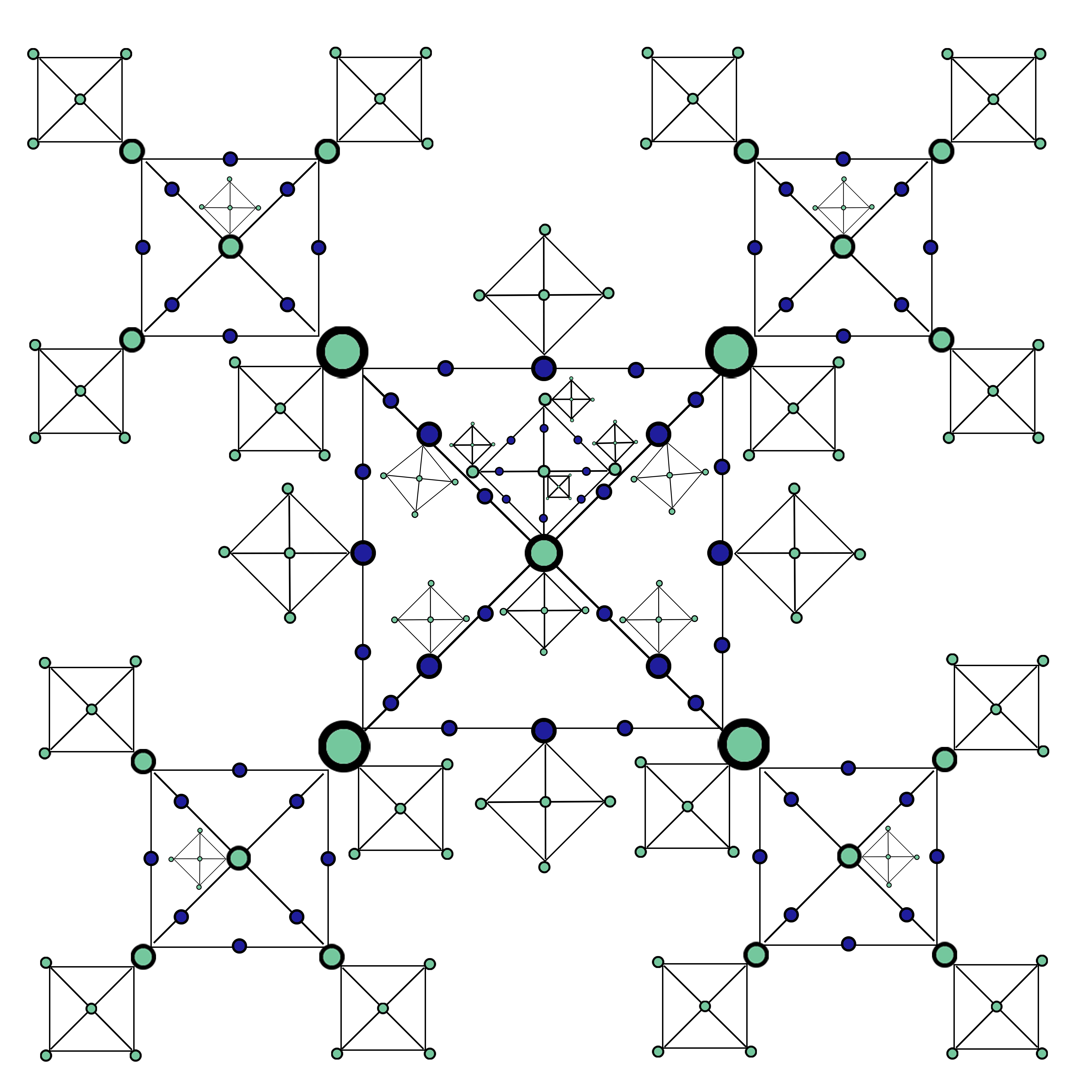}} \\
      \hline
  \end{tabular}
\end{table}

\noindent From the previous example, we can observe that $G^{(1,2,4)}$ is 5 copies of $G^{(0,2,4)}$ and $H^{(1,2)}$, and $G^{(2,2,4)}$ is 5 copies of $G^{(1,2,4)}$, $8$ copies of $G^{(0,2,4)}$ and $H^{(2,2)}$.

Therefore, each copy of $G^{(t,m,n)}$ in $G^{(i,m,n)}$ will be transformed to $G^{(t+1,m,n)}$ in $G^{(i+1,m,n)}$, and the number of copies will be preserved as in $C_n$ case. 

\begin{lemma} \label{lemma4}
    In $G^{(i,m,n)}$ of $W_n$, we have ($(n+1)$ copies of $G^{(i-1,m,n)}$, $(2n)(m-1)(m^{0})$ copies of $G^{(i-2,m,n)}$, $(2n)(m-1)(m^{1})$ copies of $G^{(i-3,m,n)}$, ..., $(2n)(m-1)(m^{(k-2)})$ copies of $G^{(i-k,m,n)}$, and $H^{(i,m)}$ which is the graph that resulted from operation $\xi_{1}$ on the $H^{(i-1,m)}$, where $H^{(1,m)}$ is the result from the operation $\xi_{1}$ on $G^{(0,m,n)}$. 
\end{lemma}
\begin{proof}
    This lemma can be proved by induction, as it is true for ($i = 1$, and $i = 2$), which is in example 6. 
    
    Assume that in $G^{(i,m,n)}$, we have ($(n+1)$ copies of $G^{(i-1,m,n)}$, $(2n)(m-1)(m^{0})$ copies of $G^{(i-2,m,n)}$, $(2n)(m-1)(m^{1})$ copies of $G^{(i-3,m,n)}$, ..., $(2n)(m-1)(m^{(k-2)})$ copies of $G^{(i-k,m,n)}$, and $H^{(i,m)}$). 
    
    And we want to prove that in $G^{(i+1,m,n)}$, we have ($(n+1)$ copies of $G^{(i,m,n)}$, $(2n)(m-1)(m^{0})$ copies of $G^{(i-1,m,n)}$, $(2n)(m-1)(m^{1})$ copies of $G^{(i-2,m,n)}$, ..., $(2n)(m-1)(m^{(k-1)})$ copies of $G^{(i-k+1,m,n)}$, and $H^{(i+1,m)}$, when we perform the operations of $\xi_{1}$ and $\xi_{2}$, we can see that each copy of $G^{(i,m,n)}$ will be converted to $G^{(i+1,m,n)}$. 

    Since each copy of $G^{(t,m,n)}$ in $G^{(i,m,n)}$ will be transformed to $G^{(t+1,m,n)}$ in $G^{(i+1,m,n)}$, and the number of copies will be preserved, and a subset of vertices in $H^{(i,m)}$ will be transformed to $G^{(0,m,n)}$, and those subsets are the vertices that have been added to the edges in the operation of $\xi_{2}$, and the $H^{(i,m)}$ will be transformed to $H^{(i+1,m)}$.

    Therefore, we can assure that in $G^{(i+1,m,n)}$ that contains ($(n+1)$ copies of $G^{(i-1,m,n)}$, $(2n)(m-1)(m^{0})$ copies of $G^{(i-2,m,n)}$, $(2n)(m-1)(m^{1})$ copies of $G^{(i-3,m,n)}$, ..., $(2n)(m-1)(m^{(k-2)})$ copies of $G^{(i-k,m,n)}$, and $H^{(i,m)}$) when we perform the operations of $\xi_{1}$ and $\xi_{2}$, we can see that each copy of $G^{(i,m,n)}$ will be converted to $G^{(i+1,m,n)}$. Moreover, we have a central graph $H^{(i,m)}$ which will grow to $H^{(i+1,m)}$, so it will grow in this way using the operations in definitions 2, and 3, to be the graph in the definition 4. 
\end{proof}

\begin{lemma} \label{lemma5}
    The number of vertices and the edges in any $W_n$ fractals that are built previously can be given by the following two recursive formulas:
    \begin{gather*}
        \vert{V^{(i)}}\vert = (n+1)\times {\vert}V^{(i-1)}\vert + (m-1)\times {\vert}E^{(i-1)}\vert \\
        {\vert}{E^{(i)}}{\vert} = (2n)\times {\vert}V^{(i-1)}{\vert} + (m)\times {\vert}E^{(i-1)}{\vert}       
    \end{gather*}
And by solving them together, it is easy to separate them as follows:
\begin{gather}
        \vert{V^{(i)}}\vert = (n + m + 1)\times {\vert}V^{(i-1)}\vert + (mn - m - 2n)\times {\vert}V^{(i-2)}\vert \\
        {\vert}{E^{(i)}}{\vert} = (n + m + 1)\times {\vert}E^{(i-1)}{\vert} + (mn - m - 2n)\times {\vert}E^{(i-2)}{\vert}    
    \end{gather}
\end{lemma}

The $n$ indicates the original graph vertices and the $m$ is the path length in the $\xi_{2}$. Therefore in fractals of $W_4$; $V^{(1)} = 5$, $E^{(1)} = 8$, and we will take the same initial step as in cycles, that $V^{(0)} = 1$, $E^{(0)} = 0$, those considered as constants. The equation form for the recursive relation in the equation of the vertices: 
\begin{equation}\label{eq4.3}
    \begin{split}
        \vert & {V^{(i)}}\vert = \frac{2^{-i-1}}{\zeta } \\
        & {\times} \left( ({\alpha_{1}}+\zeta -1) (-\zeta +\alpha_2 +1)^i+(\zeta -\alpha_2 +1) (\zeta +\alpha_2 +1)^i \right)
    \end{split}
\end{equation}
Where $\zeta = \sqrt{6 (m-1) n+(m-1)^2+n^2}$, $\alpha_1$, and $\alpha_2$ is defined as before in equation \ref{eq3.3}. ($i > 0$, $m \geq 2$, and $n \geq 3$).

\subsection{The Number of Spanning Tree for self-similar (Fractal) graphs based on $W_n$ graphs}

The number of spanning trees for $W_{n}$ can be obtained by:
\begin{equation}\label{eq4.4}
    \begin{split}
        \tau(W_{n}) & = \\
        & = L_{2n} - 2 \\
        & = F_{2n+2} - F_{2n-2} - 2 \\
        & = \left(\frac{1}{2} \left(\sqrt{5}+1\right)\right)^{2 n}+\left(\frac{2}{\sqrt{5}+1}\right)^{2 n} \cos (2 \pi  n)-2
    \end{split}
\end{equation}

\noindent And this is proved several times such as in \cite{haghighi2009recursive}. Where $L_n$ is the Lucas number, such that: $L_{n+2} = L_{n+1} + L_{n}$ for $n \geq 1$, where $L_1 = 1, L_2 = 3$. And the $F_n$ is the Fibonacci Number. 
\begin{theorem}
    The number of spanning trees for any self-similar graph based on $W_{n}$ graphs in any stage $i \geq 0$ can be given by the following equation:
    \begin{equation}\label{eq4.5}
        \begin{split}
            \tau(G^{(i,m,n)}) & = \\ 
            & {\left(L_{2n} - 2\right)}^{\sum^{i}_{j=0}{{\vert}V^{(j)}{\vert}}} \\
            & {\times}{(m^{n \times \sum^{i}_{j=0}{(i-j)\times{\vert}V^{(j)}{\vert}}})}
        \end{split}
    \end{equation}
\end{theorem}
\begin{proof}
    By Lemma \ref{lemma4}, and \ref{lemma5}, we can see that in every stage $i$, we will have $G^{i,m,n}$ such that $G^{i,m,n}$ contains ($(n+1)$ copies of $G^{(i-1,m,n)}$, $(2n)(m-1)(m^{0})$ copies of $G^{(i-2,m,n)}$, $(2n)(m-1)(m^{1})$ copies of $G^{(i-3,m,n)}$, ..., $(2n)(m-1)(m^{(k-2)})$ copies of $G^{(i-k,m,n)}$, and $H^{(i,m)}$, and in the $G^{(i-1,m,n)}$ ($(n+1)$ copies of $G^{(i-2,m,n)}$, $(2n)(m-1)(m^{0})$ copies of $G^{(i-3,m,n)}$, $(2n)(m-1)(m^{1})$ copies of $G^{(i-4,m,n)}$, ..., $(2n)(m-1)(m^{(k-2)})$ copies of $G^{(i-k-1,m,n)}$, and $H^{(i-1,m)}$. Recursively, we will have every time a specific number of $G^{(0,m,n)}$, which is the ordinary wheel graph which has a specific formula by which we can know the number of spanning trees in equation \ref{eq4.4}, and it will be repeated $\vert V^{(i)}$ and this $i$ vary from 0 to the $i$ stage itself, the same case with the cycles. The rest part will be the central graphs $H(i,m)$, which can be calculated through that, every $G^{(i,m,n)}$ contains $H(i,m)$, and number of $G^{(i-1,m,n)}$ that contains $H^{(i-1,m)}$, and so on, till having number of $H^{(1,m)}$, number of $H^{(2,m)}$, ..., and one $H^{(i,m)}$. And $H^{(1,m)}$ can be calculated easily through this relation $\tau (H^{(1,m)}) = m^{n}$, as it is not a wheel anymore, and it can be noticed easily by considering the first wheel graph as a join between $C_{n}$ graph and $K_{1}$ as here \cite{article}, then we have n number of $C_n$ which evolved to be $C_{m^{i}}$, so the number of spanning trees for the central graph is $\tau (H^{(1,m)}) = m^{n \times i}$. Therefore, the second part of the formula \ref{eq4.5} is for the central graphs. 
\end{proof}

And by executing the summations, using the Binet formula in equation \ref{eq4.3}, the first summation will be as follows: 
\begin{equation}
    \begin{split}
        & \sum _{j=0}^i V^{(j)} = \frac{2^{-i-1}}{\zeta n} {\times} \\
        & (\zeta  \left(2^{i+1}+(n-1) (\eta +\Omega )\right) \\
        & +(\Omega -\eta ) (m (n-1)-n (n+4)+1))
    \end{split}
\end{equation}

\begin{equation}
    \begin{split}
        & \sum _{j=0}^i (i-j)V^{(j)} = \frac{2^{-i-1}}{\zeta n (m-1)} {\times} \\
        & (2^{-i-1}(\zeta  2^{i+1} (m (i n+n-1)-((i+3) n)+1)) \\
        & + (\eta -\Omega ) \left((3 m-5) n^2+(m-6) (m-1) n-(m-1)^2\right) \\
        & - \zeta  (\eta +\Omega ) (m (n-1)-3 n+1))
    \end{split}
\end{equation}

Where $(\zeta +\alpha_2 +1)^i = \eta$, and $(- \zeta +\alpha_2 +1)^i = \Omega$. 

\noindent \textbf{Example 7}: If we are going to take $n = 4$ and $m = 2$ as  Example 6. We will find that, the number of vertices in each stage $i$ beginning from 0 is as follows: $(1, 5, 33, 221, 1481, ...)$, the number of spanning trees (knowing that the number spanning trees of $W_4$ can be obtained by the formula \ref{eq4.4} and it is equal to 45):
\begin{equation*}
\begin{split}
\tau\left(G^{\left(1,2,4\right)}\right) & = 45^{{\vert}V^{1}{\vert}+{\vert}V^0{\vert}}\times2^{4 {\times} ({0\times {\vert}V^{1}{\vert} + 1 \times {\vert}V^0{\vert}})} \\
 & = 45^{5+1}\times2^4 \\ 
 & = 45^6\times2^4
\end{split}
\end{equation*}
\begin{equation*}
\begin{split}
\tau\left(G^{\left(2,2,4\right)}\right) & = 45^{{\vert}V^2{\vert}+{\vert}V^1{\vert}+{\vert}V^0{\vert}}\times2^{ 4 {\times}({\vert}V^1{\vert}+2{\vert}V^0{\vert})} \\
 & = 45^{33+5+1}\times2^{4(5+2)} \\
 & = 45^{39}\times2^{28}
\end{split}
\end{equation*}
\begin{equation*}
\begin{split}
\tau\left(G^{\left(3,2,4\right)}\right) & = 45^{{\vert}V^3{\vert}+{\vert}V^2{\vert}+{\vert}V^1{\vert}+{\vert}V^0{\vert}}\times2^{4 {\times}({\vert}V^2{\vert}+2{\vert}V^1{\vert}+3{\vert}V^0{\vert})} \\ 
 & = 45^{221+33+5+1}\times2^{4(33+10+3)} \\
 & = 45^{260}\times2^{184}
\end{split}
\end{equation*}

This confirms the computational results that are based on taking a graph's vertices and edges and computing the determinant of the Laplacian matrix using a Python program. 

\subsection{The Entropy in self-similar graphs (Fractals) based on $W_n$ graphs}
\begin{corollary}
    By using the definition of entropy and its formula in \ref{eq3.7}, we will get the entropy $\sigma(G^{(i,m,n)})$ but after defining some variables: 
\begin{equation*}
    \begin{split}
    & A = \frac{4 (m-1)}{(\zeta -\alpha_2 +1) (\zeta +\alpha_2 -1)^2} \\
    & B = \frac{4 n \ln (m)}{(\zeta -\alpha_2 +1)^2} \\ 
        & \times (-\zeta +m^2 (n-1) + m (\zeta +n (-\zeta +3 n-7)+2) \\
        & +n (3 \zeta -5 n+6)-1) \\
    & C = \frac{1}{-\zeta +\alpha_2 -1} \\
        & \times (\zeta +\alpha_2 -1) (\zeta +m (n-1)-n (\zeta +n+4)+1) \\
        & \times \text{ln}\left[\left(\frac{1}{2} \left(\sqrt{5}+1\right)\right)^{2 n}-2\right]. \\
    & D = 4^n \left(\sqrt{5}+1\right)^{-2 n} \cos (2 \pi  n)
    \end{split}
\end{equation*}
\begin{equation}
    \sigma(G^{(i,m,n)}) = A {\times} (B + C + D)
\end{equation}
With conditions that $n\geq 3$, and $m\geq 2$.

And plot in \ref{Fig. 7}, with $n$ are on the x-axis and $m$ is on y-axis, z-axis is the entropy.
\end{corollary}
\begin{figure}[h]
    \centering
    \includegraphics[width=0.15\textwidth, height=0.15\textwidth]{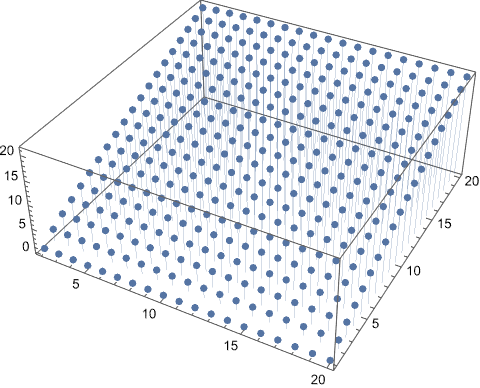}
    \caption{Plot of entropy versus m and n.}
    \label{Fig. 7}
\end{figure}

\newpage

\subsection{The Clustering Coefficient for self-similar graphs (Fractals) based on $W_n$ graphs}

\begin{lemma}
    Using the definition in the previous section and formula \ref{eq3.9}, it is easy to obtain a formula for the average clustering coefficients for wheel graphs $W_n$. 

    \begin{equation} \label{eq:4.9}
    \bar{\mathcal{A}}(W_{(n)}) = \frac{1}{n + 1} \left(\frac{2n}{3} + \frac{2}{n - 1}\right)
    \end{equation}
\end{lemma}
\begin{proof}
    As wheels can be considered as $C_n$ joined with $K_1$, therefore we can assure that all vertices always will take degrees equal to $3$ and there will be just 2 connections with the neighbors except in the case of $W_3$. 

As the degree comes from the connection with the central vertex and the 2 vertices in the neighborhood, those 2 vertices by default are not connected and they are just connected with the third one which is the central. Therefore the degree is 3 and the connection between the neighbors is 2. 

The degree for any central graph will be $n$ and the connection between the neighborhood will be $n$ as well, the cycle that connects all the other vertices except the central one. 

For so, it is possible to ensure that the formula \ref{eq:4.9} is True.
\end{proof}
\noindent \textbf{Example 8:} Taking $W_{4}$ for example: 
\begin{figure}[h]
    \centering
    \includegraphics[width=0.15\textwidth, height=0.15\textwidth]{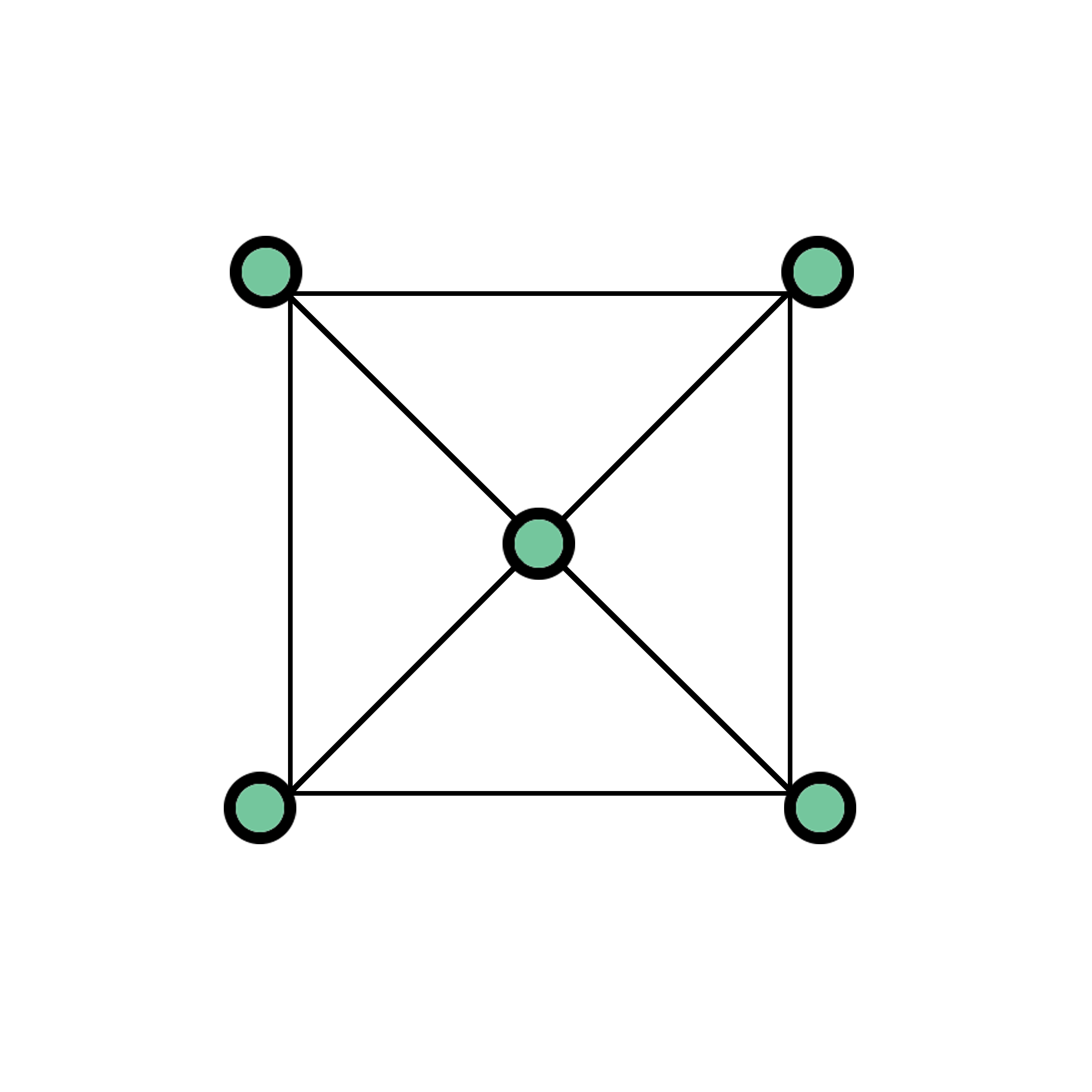}
    \caption{$W_4$ Graph}
    \label{fig: Wheel of 4}
\end{figure}

Therefore: $\bar{\mathcal{A}}(W_{4}) =\frac{1}{4 + 1} \left(\frac{2\times 4}{3} + \frac{2}{4 - 1}\right) = \frac{2}{3}$ Which is the correct and the same answer using a Python program. 
\begin{lemma} \label{lemma:6}
    The average clustering coefficient for any self-similar graph based on graphs of $W_n$ where $n \geq 3$ in stages $i = 1$ is as follows: 
    \begin{equation} \label{eq: 4.10}
        \bar{\mathcal{A}}(G^{1,m}_{W_n}) = \frac{\frac{2(n-1)\vert V^{1}\vert}{\binom{3}{2}} + \frac{n\vert V^{1}\vert}{\binom{n}{2}} + \frac{2n}{\binom{6}{2}} + \frac{2\times 1}{\binom{(n+3)}{2}}}{\vert V^{2}\vert}
    \end{equation}
\end{lemma}

\begin{proof}
    As previously stated the wheels can be considered as $C_{n}$ joined with $K_{1}$, therefore the degree of each vertex will be transformed as follows: 
    \begin{center}
        \begin{tabular}{ c c c c c c c c c }
            $0$ & $\rightarrow$ & $1$ & $\rightarrow$ & $2$ & $\rightarrow$ & $...$ & $\rightarrow$ & $i$ \\
            $0$ & $\rightarrow$ & $2$ & $\rightarrow$ & $5$ & $\rightarrow$ & $...$ & $\rightarrow$ & $2 + 3(i-1)$ \\ 
            $3$ & $\rightarrow$ & $6$ & $\rightarrow$ & $9$ & $\rightarrow$ & $...$ & $\rightarrow$ & $3 + 3(i)$\\
            $n$ & $\rightarrow$ & $n+3$ & $\rightarrow$ & $n+6$ & $\rightarrow$ & $...$ & $\rightarrow$ & $n + 3(i)$
        \end{tabular}
    \end{center}
    As the zero degrees is a vertex that does not exist according to the definition of the operation $\xi_{1}$, then it is developed to be 2 according to applying $\xi_{1}$ to each edge, and then we attach them to $G^{(0,m,n)}$, so we attach each one to another 3 vertices, and so on. 3 is the original degree for any vertex and not the central vertex in any wheel, and as we apply $\xi_{2}$ we attach it to another 3 vertices, and so on. The central vertex has a degree of n and similarly will increase by 3 each time due to $\xi_{2}$. 
    Hence we know that each vertex contributes another $n$ vertex, so $(n-1)$ of them will have the degree 3 and with only 2 connections between the neighborhood, and $n$ of them will be the central vertex with degree $n$, the original vertices in $G^{(0,m,n)}$ are represented in the last two terms, as n of them will have the doubled degree as it attached to another 3 vertices from the new graph and the edges of the previous graph was replaced by the path of length $m$, with $(m-1)$ vertices in between the original vertices, therefore the overall degree will be 6 and with only 2 connections between the neighborhood. The last term is the original central vertex with a degree (n + 3) similar to the previous way, as it attached to another 3 vertices from the new graph, and the edges of the previous graph were replaced by the path of length $m$, with $(m-1)$ vertices in between the original vertices, therefore the overall degree will be (n + 3) with just 2 connections between the neighborhood. Finally, we have in $G^{(1,m,n)}$, $V^{(2)}$ vertex. 
\end{proof}
\noindent \textbf{Example 9:} Taking $W_5$,
\begin{figure}[h]
    \centering
    \includegraphics[width=0.15\textwidth, height=0.15\textwidth]{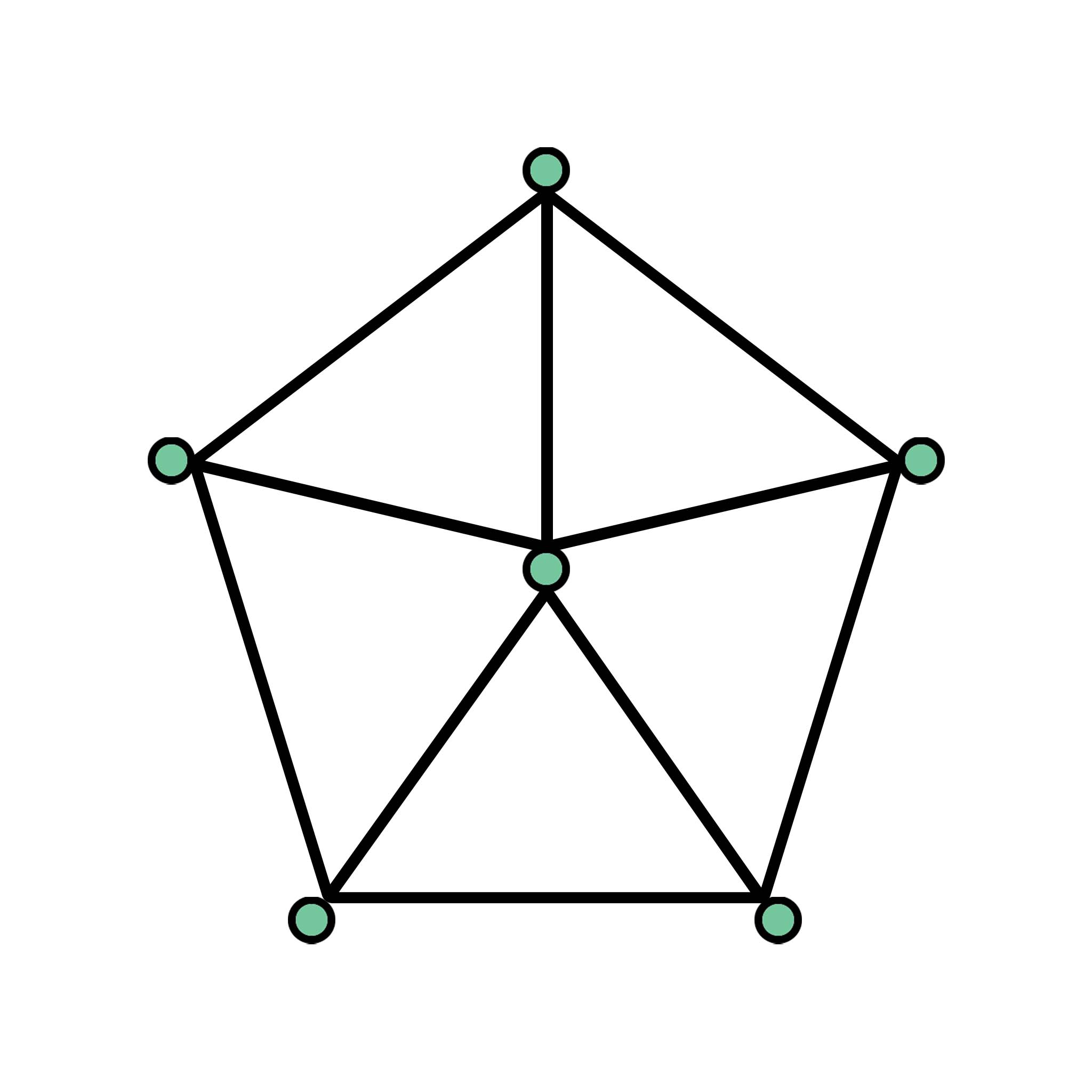}
    \caption{$G^{(0,m,5)}$ Graph}
    \label{fig:W5(0)}
\end{figure}
\begin{figure}[h]
    \centering
    \includegraphics[width=0.3\textwidth, height=0.3\textwidth]{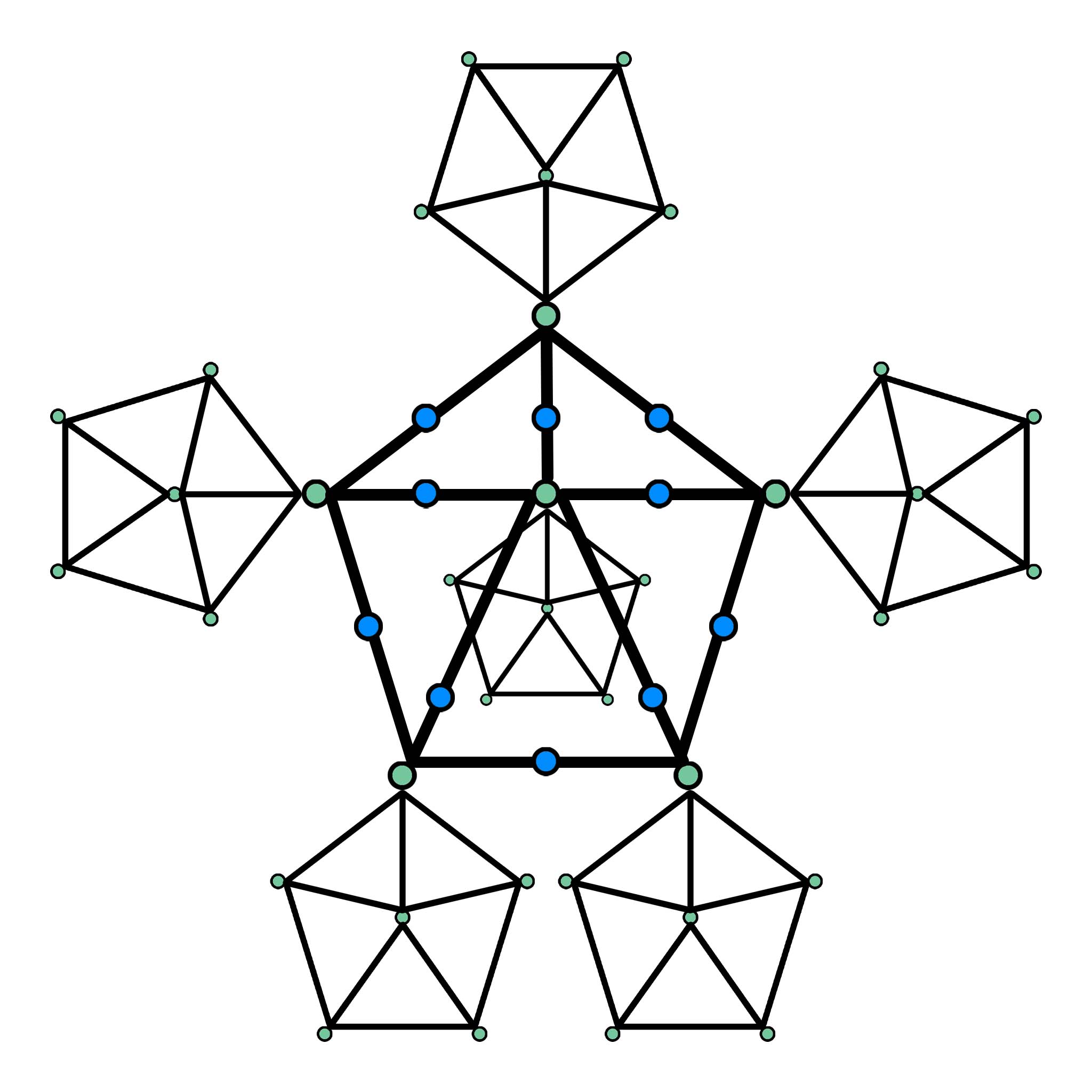}
    \caption{$G^{(1,2,5)}$ Graph}
    \label{fig:W5(1)}
\end{figure}

So, $\bar{\mathcal{A}}(G^{1,2}_{W_5}) = \frac{\frac{2(5-1)\vert V^{1}\vert}{\binom{3}{2}} + \frac{5\vert V^{1}\vert}{\binom{5}{2}} + \frac{2\times 5}{\binom{6}{2}} + \frac{2\times 1}{\binom{(5+3)}{2}}}{\vert V^{2}\vert} = \frac{815}{1932}$ Which is the correct and the same answer using a Python program. 

\begin{theorem}
    The average clustering coefficient for any self-similar graph based on graphs of $W_n$, such that $n \geq 3$, and in any stage $i \geq 1$ is as follows: 
    \begin{equation}
        \bar{\mathcal{A}}(G^{(i,m,n)}) = \frac{{\varsigma_{1}} + {\varsigma_{2}} + {\varsigma_{3}} + {\varsigma_{4}} + {\varsigma_{5}}}{\vert V^{(i + 1)} \vert}
    \end{equation}
    Where the constants are as follows: 
    \begin{equation*}
        \begin{split}
            {\varsigma_{1}} &= \frac{n\times 2}{\binom{3(i+1)}{2}} \\
            {\varsigma_{2}} &= \frac{n\times \vert V^{i}\vert}{\binom{n}{2}} \\
            {\varsigma_{3}} &= \sum^{i-1}_{j = 0} \frac{2(n-1)\vert V^{(i-j)} \vert}{\binom{3(j+1)}{2}} \\
            {\varsigma_{4}} &= \sum^{i}_{j = 1} \frac{2 \vert V^{(i-j)} \vert}{\binom{(3j+n)}{2}} \\
            {\varsigma_{5}} &= \sum^{i-1}_{j = 1} \frac{2(m-1) \vert E^{(i-j)} \vert}{\binom{(3j+2)}{2}}
        \end{split}
    \end{equation*}
\end{theorem}
\begin{proof}
    Using lemma \ref{lemma:6} and it's proof that the degree of each vertex will be transformed. Therefore in each $i$ stage, we have another 3 degrees added to the set of degrees that we already have. So for example in stage 2 with $n = 5$, we will have the following set of degrees: \{0, 2, 5, 3, 6, 9, 5, 8, 11\}, consequently, the number of vertices is preserved and they are just transforming the degree by the nature of the two operations $\xi_{1}$ and $\xi_{2}$. 
    
    As the graph $G^{(0,m,n)}$ begins with $(n+1)$ vertices, $n$ vertex with degree $3$ and the central vertex with degree $n$. then with $G^{(1,m,n)}$, each vertex has been added due to $\xi_{1}$ will have degree $2$, and each vertex other than the central vertices have been added due to $\xi_{2}$ will have degree $3$, and the central vertices have been added due to $\xi_{2}$ will have $n$ degree. Furthermore, the original $n$ vertices will have a degree of 6, and the original central vertex will have a degree of $(n + 3)$. And so on. Furthermore, plugging $i = 1$, will give us the formula in lemma \ref{lemma:6}.
\end{proof}

\bibliography{Bibliography}

\end{document}